\documentclass{amsart}

\usepackage{macros}
\standardsettings

\draftfalse

\newcommand{\msp}{\medskip}
\newcommand{\fr}{\noindent}

\newcommand{\rad}{{\mathrm r}}

\newcommand{\hD}{\dim_{\mathrm{hyp}}}
\newcommand{\hDS}{\dim_{\mathrm{IFS}}}
\newcommand{\cD}{\dim_{\mathrm{conf}}}

\newcommand\h{{\text h}}
\newcommand\hmu{\h_\mu}
\newcommand\htop{{\text h_{\text{top}}}}

\begin{document}

\title[]
{\bf{{\large {B}}adly approximable vectors and fractals \\ defined by conformal dynamical systems}}
\date{\today}

\authortushar\authorlior\authordavid\authormariusz


\begin{abstract}
We prove that if $J$ is the limit set of an irreducible conformal iterated function system (with either finite or countably infinite alphabet), then the badly approximable vectors form a set of full Hausdorff dimension in $J$. The same is true if $J$ is the radial Julia set of an irreducible meromorphic function (either rational or transcendental). The method of proof is to find subsets of $J$ that support absolutely friendly and Ahlfors regular measures of large dimension. In the appendix to this paper, we answer a question of Broderick, Kleinbock, Reich, Weiss, and the second-named author ('12) by showing that every hyperplane diffuse set supports an absolutely decaying measure.
\end{abstract}

\maketitle


\section{Introduction}

Fix an integer $d\ge 1$, and recall that a point (vector) $\xx\in \R^d$ is said to be \emph{badly approximable} if there exists $c > 0$ such that for any $\pp\in\Z^d$ and $q\in\N$, one has
\[
\|q\xx-\pp\|\ge c/q^{1/d}.
\]
Denote the set of all badly approximable vectors in $\R^d$ by $\BA_d$. It is well-known that Lebesgue measure of $\BA_d$ is zero, but nevertheless this set is quite large: its Hausdorff dimension is equal to $d$. When $d = 1$, a number is badly approximable if and only if the partial quotients of its continued fraction expansion are uniformly bounded. An analogous result when $d > 1$ is the Dani--Kleinbock--Margulis correspondence principle \cite{Dani4, KleinbockMargulis}, according to which a vector is badly approximable if and only if a certain trajectory in the homogeneous space $\SL_{d + 1}(\R)/\SL_{d + 1}(\Z)$ is bounded. However, this result does not allow one to compute explicit examples of badly approximable vectors in the same way that one can write down arbitrary continued fraction expansions with bounded partial quotients, and in fact, very few explicit examples of badly approximable vectors in higher dimensions are known.

Recently, there has been a growing interest in computing the Hausdorff dimension of the interesection of $\BA_d$ with various fractal sets. Since $\BA_d$ has full dimension, one expects its intersection with any fractal set $J \subset \R^d$ to have the same dimension as $J$, and this can be proven for certain broad classes of fractal sets $J$. For example, Kleinbock and Weiss proved the following:

\begin{theorem}[{\cite[Theorem 1.1]{KleinbockWeiss1}}]
\label{theoremkleinbockweiss}
Let $\mu$ be a finite measure\Footnote{In this paper, all measures are assumed to be Borel.} on $\R^d$ that is absolutely friendly (see below for the definition). Then
\[
\HD(\BA_d\cap\Supp(\mu))\ge \inf_{\xx\in\Supp(\mu)} \underline d_\mu(\xx),
\]
where $\underline d_\mu(\xx)$ denotes the lower pointwise dimension of $\mu$ at $\xx$, i.e.
\[
\underline d_\mu(\xx) \df \liminf_{r\to 0} \frac{\log \mu(B(\xx,r))}{\log(r)},
\]
and $\Supp(\mu)$ denotes the topological support of $\mu$. In particular, if $\mu$ is Ahlfors regular of dimension $\delta$,\Footnote{We recall that $\mu$ is said to be \emph{Ahlfors regular of dimension $\delta$} if there exists a constant $C\geq 1$ such that for all $\xx\in\Supp(\mu)$ and $r\in\OC 01$, we have $C^{-1} r^\delta \leq \mu(B(\xx,r)) \leq C r^\delta$.} then $\HD(\BA_d\cap\Supp(\mu)) = \delta = \HD(\Supp(\mu))$.\Footnote{This last sentence was proven independently by Kristensen, Thorn, and Velani \cite[Theorem 8]{KTV}.}
\end{theorem}

\fr We recall that a measure $\mu$ is called \emph{doubling} (or \emph{Federer}) if there exists a constant $C > 0$ such that for all $\xx\in\R^n$ (equiv. for all $\xx \in \Supp(\mu)$) and for all $r > 0$, we have
\[
\mu(B(\xx,2r))\le C\mu(B(\xx,r)).
\]
Furthermore, $\mu$ is called \emph{absolutely decaying} if there exist $C, \alpha, r_0 > 0$ such that for all $\xx\in\Supp(\mu)$, $0 < r \leq r_0$, and $\varepsilon > 0$, and for every affine hyperplane $\LL\subseteq\R^n$, we have
\[
\label{absolutelydecaying}
\mu\left(\NN(\LL,\varepsilon r)\cap B(\xx,r)\right) \leq C\beta^\alpha\mu(B(\xx,r)),
\]
where 
\[
\NN(\LL,\varepsilon r) \df \{\xx\in \R^d : d(\xx,\LL) \leq \varepsilon r\}
\] 
is the closed neighborhood of $\LL$ of thickness $\varepsilon r$. Finally, $\mu$ is called \emph{absolutely friendly} if it is both doubling and absolutely decaying.

\begin{remark}
More generally, Theorem \ref{theoremkleinbockweiss} remains true if $\BA_d$ is replaced by any \emph{hyperplane absolute winning (HAW)} set; see \cite{BFKRW} for the definition. The class of hyperplane absolute winning set includes many sets coming from dynamics and number theory, see e.g. \cite{AGK,BFKRW,FKMS1,NesharimSimmons}. Such a generalization can be proven by combining \cite[Proposition 5.1]{KleinbockWeiss2} with \cite[Propositions 4.7 and 5.1]{BFKRW}. The same generalization is valid for all results in this paper. Note that according to \cite[\62]{BFKRW}, if $(f_i)_1^\infty$ is a sequence of $\CC^1$ diffeomorphisms then the set
\[
\bigcap_{i = 1}^\infty f_i^{-1}(\BA_d)
\]
is hyperplane absolute winning. Thus all of our theorems could be strengthened by replacing $\BA_d$ by this set.
\end{remark}

Several recent results involve verifying the hypotheses of Theorem \ref{theoremkleinbockweiss} for measures supported on the limit sets of various conformal dynamical systems, namely iterated function systems, Kleinian groups, and rational functions:

\begin{theorem}[{\cite[Corollary 1.6]{Urbanski}} and {\cite[Lemma 3.14]{MauldinUrbanski1}}]
\label{theorem1}
Let $J$ be the limit set of a finite conformal IFS on $\R^d$, and suppose that $J$ is not contained in any real-analytic hypersurface of $\R^d$. Let $\delta = \HD(J)$. Then the $\delta$-dimensional Hausdorff measure restricted to $J$ is absolutely friendly and Ahlfors regular of dimension $\delta$.
\end{theorem}


\begin{theorem}[{\cite[Theorem 2]{StratmannUrbanski1}} and {\cite[Theorem 2]{StratmannVelani}}]
\label{theorem2}
Let $\mu$ be the Patterson--Sullivan measure of a convex-cocompact Kleinian group, and suppose that the limit set $J = \Supp(\mu)$ is not contained in a generalized sphere.\Footnote{Recall that a \emph{generalized sphere} is either a sphere or the union of $\{\infty\}$ with an affine hyperplane.} Let $\delta = \HD(J)$. Then $\mu$ is absolutely friendly and Ahlfors regular of dimension $\delta$.
\end{theorem}

\begin{theorem}[{\cite[Theorem 1.10]{DFSU_GE2}} and {\cite[Theorem 4 and its Corollary]{Sullivan_conformal_dynamical}}]
\label{theorem3}
Let $f:\what\C\to\what\C$ be a hyperbolic (i.e. expansive on its Julia set) rational function, let $J$ be its Julia set, and suppose that $J$ is not contained in any generalized sphere. Let $\delta = \HD(J)$. Then the $\delta$-dimensional Hausdorff measure restricted to $J$ is absolutely friendly and Ahlfors regular of dimension $\delta$.
\end{theorem}

\noindent In all three of these cases, Theorem \ref{theoremkleinbockweiss} implies that $\HD(\BA_d\cap J) = \HD(J)$. In the case of Theorem \ref{theorem1}, we state this corollary for future reference:

\begin{corollary}
\label{corollary1}
With notation as in Theorem \ref{theorem1}, we have $\HD(\BA_d\cap J) = \HD(J)$.
\end{corollary}

Theorems \ref{theorem1}-\ref{theorem3} are reasonably optimal if one's goal is to show that the limit set $J$ of a given conformal dynamical system supports an absolutely friendly and Ahlfors regular measure. (This is not quite true; see Appendix \ref{appendixsemihyperbolic} for a generalization of Theorem \ref{theorem3} where ``hyperbolic'' is replaced by ``semi-hyperbolic''.) But if one restricts to the more modest goal of showing that $\HD(\BA_d\cap J) = \HD(J)$, then they can be vastly generalized. The idea is that although $J$ may not be equal to the topological support of an absolutely friendly and Ahlfors regular measure $\mu$, it may be possible to find such a measure $\mu$ whose support is a subset of $J$. In this case Theorem \ref{theoremkleinbockweiss} implies that $\HD(\BA_d\cap J) \geq \HD(\mu)$, and so by finding a sequence of measures $(\mu_n)_1^\infty$ such that $\HD(\mu_n) \nearrow \HD(J)$ we can show that $\HD(\BA_d\cap J) = \HD(J)$. Three of the authors used this idea in a previous paper to prove the following result in the setting of Kleinian groups:

\begin{theorem}[{\cite[Theorem 9.3]{FSU4}}]
\label{theoremFSU}
Let $J_\rad$ be the radial limit set of a Kleinian group, and suppose that $J_\rad$ is not contained in a generalized sphere. Then $\HD(\BA_d\cap J_\rad) = \HD(J_\rad)$.
\end{theorem}

\begin{remark}
Theorem \ref{theoremFSU} shows that for a geometrically finite group, we have $\HD(\BA_d\cap J) = \HD(J)$ unless $J$ is contained in a generalized sphere. In general, the Patterson--Sullivan measures of such groups are not necessarily absolutely friendly; cf. \cite[Theorem 1.9]{DFSU_GE2}.
\end{remark}

In the theory of conformal dynamical systems it is often possible to prove analogues of theorems regarding Kleinian groups in the setting of rational functions, and vice-versa, a phenomenon known as ``Sullivan's dictionary'' (see e.g. \cite{McMullen_classification}). Sullivan's dictionary has recently been extended to the realm of conformal iterated function systems, see \cite[Table 1]{DSU_rigidity}. In this paper, using the same idea of approximating the limit set from below by absolutely friendly and Ahlfors regular measures, we prove analogues of Theorem \ref{theoremFSU} in the sense of an extended Sullivan's dictionary. Specifically, we will prove the following:

\begin{theorem}[See Theorem \ref{maintheorem1}]
If $J$ is the limit set of an irreducible conformal iterated function system with a countable alphabet, then $\HD(\BA_d\cap J) = \HD(J)$.
\end{theorem}
\begin{theorem}[See Theorem \ref{maintheorem2}]
If $J_\rad$ is the radial Julia set of an irreducible meromorphic dynamical system (i.e. a rational function or a transcendental meromorphic function), then $\HD(\BA_2\cap J_\rad) = \HD(J_\rad)$.
\end{theorem}

Other significant results include Corollary \ref{BA_GDMS} (extending Theorem \ref{maintheorem1} to certain graph directed Markov systems), Theorem \ref{All_Dim_=} (proving that several dynamically defined quantities are all equal to $\HD(J_\rad)$, indicating that it is the ``natural'' dynamical dimension of a meromorphic dynamical system), and Corollaries \ref{c5_2015_10_17} and \ref{maincorollary} (describing some special cases in which $J_\rad$ has the same Hausdorff dimension as the total Julia set $J$, and thus $\HD(\BA_2\cap J) = \HD(J)$). Finally, in the appendix we prove that every hyperplane diffuse set supports an absolutely decaying measure, thus answering a question of \cite{BFKRW}. We conclude with a list of open questions.\\


{\bf Acknowledgements.} The first-named author was supported in part by a 2016-2017 Faculty Research Grant from the University of Wisconsin--La Crosse. The second-named author was supported in part by the Simons Foundation grant \#245708. The third-named author was supported in part by the EPSRC Programme Grant EP/J018260/1. The fourth-named author was supported in part by the NSF grant DMS-1361677. The authors thank Barak Weiss and the anonymous referee for helpful comments.

\section{Conformal graph directed Markov systems \\ and \\ iterated function systems}

\msp We now proceed to define iterated function systems (abbr. IFSes) and their generalizations, graph directed Markov systems (abbr. GDMSes). Such systems are studied at length in \cite{MauldinUrbanski1} and \cite{MauldinUrbanski2}, respectively. A \emph{directed multigraph} consists of a finite set $V$ of vertices, a countable (either finite or infinite) set $E$ of directed edges, and two functions $i,t:E\to V$. An \emph{incidence matrix} on $(V,E,i,t)$ is a map $A:E\times E\to \{0,1\}$ such that for all $e,f\in E$ with $A_{ef} = 1$, we have $t(f) = i(e)$. Now suppose that in addition, we have a collection of nonempty compact metric spaces $\{X_v\}_{v\in V}$ and a number $\lambda \in (0,1)$, and that for every $e\in E$, we have a one-to-one contraction $\phi_e:X_{t(e)}\to X_{i(e)}$ with Lipschitz constant $\le \lambda$. Then the collection
\[
\SS = \{\phi_e:X_{t(e)}\to X_{i(e)}\}_{e\in E}
\]
is called a \emph{graph directed Markov system} (or \emph{GDMS}). We now describe the limit set of the system $\SS$. For every $n \in \N$ let
\[
E_A^n\df\{\omega\in E^n:\forall (1\le j\le n-1) \;\; A_{\omega_j\omega_{j+1}}=1 \},
\]
and note that $E_A^0$ is the set consisting of the empty word. Then consider the disjoint union
\[
E_A^*\df\bigcup_{n=0}^\infty E_A^n
\]
and let
\[
E_A^\infty\df\{\omega\in E^\infty:\text{every finite subword of $\omega$ is in $E_A^*$}\}.
\]
For each $\omega\in E_A^*$, we let $|\omega|$ denote the unique integer $n$ such that $\omega\in E_A^n$; we call $|\omega|$ the \emph{length} of $\omega$. For each $\omega\in E_A^\infty$ and $n \in \N$, we write
\[
\omega|_n\df\omega_1\omega_2\ldots\omega_n\in E_A^n.
\]
For each $n \geq 1$ and $\omega \in E_A^n$, we let $i(\omega) = i(\omega_1)$ and $t(\omega) = t(\omega_n)$, and we let
\[
\phi_\omega=\phi_{\omega_1}\circ\cdots\circ\phi_{\omega_n}:X_{t(\omega)}\to X_{i(\omega)}.
\]
For each $\omega \in E^\infty_A$, the sets $\{\phi_{\omega|_n}\left(X_{t(\omega_n)}\right)\}_{n \geq 1}$ form a descending sequence of nonempty compact sets and therefore $\bigcap_{n \geq 1}\phi_{\omega|_n}\left(X_{t(\omega_n)}\right)\ne\emptyset$. Since for every $n \geq 1$, $\diam\left(\phi_{\omega|_n}\left(X_{t(\omega_n)}\right)\right)\le \lambda^n\diam\left(X_{t(\omega_n)}\right)\le \lambda^n\max\{\diam(X_v):v\in V\}$, we conclude that the intersection 
\[
\bigcap_{n \in \N}\phi_{\omega|_n}\left(X_{t(\omega_n)}\right)
\]
is a singleton, and we denote its only element by $\pi(\omega)$. In this way we have defined a map
\[
\pi:E^\infty_A\to \coprod_{v\in V}X_v,
\]
where $\coprod_{v\in V} X_v$ is the disjoint union of the compact sets $X_v \; (v\in V)$. The map $\pi$ is called the \emph{coding map}, and the set
\[
J = J_\SS \df \pi(E^\infty_A)
\]
is called the \emph{limit set} of the GDMS $\SS$. The sets
\[
J_v = \pi(\{\omega \in E_A^\infty : i(\omega_1) = v\}) \;\;\;\;\; (v\in V)
\]
are called the \emph{local limit sets} of $\SS$.

\msp We call a GDMS $\SS$ \emph{finite} if its alphabet $E$ is finite. Furthermore, we call $\SS$ \emph{maximal} if for all $e,f\in E$, we have $A_{ef}=1$ if and only if $t(f)=i(e)$.  In \cite{MauldinUrbanski2}, a maximal GDMS was called a \emph{graph directed system} (abbr. GDS).
Finally, we call $\SS$ an \emph{iterated function system} (or \emph{IFS}) if $\SS$ is maximal and $V$, the set of vertices of $\SS$, is a singleton. Equivalently, $\SS$ is an IFS if the set of vertices of $\SS$ is a singleton and all entries of the incidence matrix $A$ are equal to $1$.

\begin{definition}\label{definitionsymbolirred}
We call the GDMS $\SS$ and its incidence matrix $A$ \emph{finitely (symbolically) irreducible} if there exists a finite set $\Omega\subset E_A^*$ such that for all $e,f\in E$ there exists a word $\omega\in\Omega$ such that the concatenation $e\omega f$ is in $E_A^*$. $\SS$ and $A$ are called \emph{finitely primitive} if the set $\Omega$ may be chosen to consist of words all having the same length. If such a set $\Omega$ exists but is not necessarily finite, then $\SS$ and $A$ are called \emph{(symbolically) irreducible} and \emph{primitive}, respectively. Note that all IFSes are symbolically irreducible.
\end{definition}

Intending to pass to geometry, we call a GDMS \emph{conformal} if for some $d\in\N$, the following conditions are satisfied (cf. \cite[\64.2]{MauldinUrbanski2}):

\msp
\begin{itemize}
\item[(a)] For every vertex $v\in V$, $X_v$ is a compact connected subset of $\R^d$, and $X_v=\overline{\Int(X_v)}$.
\item[(b)] (Open Set Condition) For all $e,f\in E$ such that $e\ne f$,
\[
\phi_e(\Int(X_{t(e)}))\cap \phi_f(\Int(X_{t(f)}))=\emptyset.
\]
\item[(c)] There exist open connected sets $W_v \supset X_v \;\; (v\in V)$ such that for every $e\in E$, the map $\phi_e$ extends to a $C^1$ conformal diffeomorphism from $W_{t(e)}$ into $W_{i(e)}$ with Lipschitz constant $\leq \lambda$.

\item[(d)] There are two constants $L\ge 1$ and $\alpha > 0$ such that for every $e\in E$ and every two points $\xx,\yy\in X_{t(e)}$,
\[
\bigl| |\phi_e'(\yy)|-|\phi_e'(\xx)| \bigr| \le L\|(\phi_e')^{-1}\|^{-1}\|\yy-\xx\|^\alpha,
\]
where $|\phi_\omega'(\xx)|$ denotes the scaling factor of the similarity matrix $\phi_\omega'(\xx)$.
\end{itemize}

\begin{remark}[{\cite[Proposition 4.2.1]{MauldinUrbanski2}}]
\label{p1.033101}
If $d\ge 2$ and a family $\SS = \{\phi_e\}_{e\in E}$ satisfies the conditions (a) and (c), then it also satisfies condition (d) with $\alpha=1$.
\end{remark}

\begin{remark}
Our definition of a conformal GDMS differs from that of \cite{MauldinUrbanski2} in that we do not assume the cone condition (i.e. \cite[(4d) on p.72]{MauldinUrbanski2}). This is important because it will make it easier to prove the existence of certain conformal IFSes in Section \ref{sectionmeromorphic}. Most of the results in \cite{MauldinUrbanski2} (and in particular the ones we cite below, namely \cite[Theorems~4.2.11~and~4.2.13]{MauldinUrbanski2}) do not depend on the cone condition. To see this, it is helpful to illustrate how the key lemma \cite[Lemma 4.2.6]{MauldinUrbanski2} can be proven without using the cone condition. We sketch a proof as follows: First without loss of generality suppose that all elements of $F$ have comparable diameters. Then estimate from below the volumes of the sets $(\phi_\omega(X_{t(\omega)}))_{\omega\in F}$ using the bounded distortion property. Finally, use the fact that these sets are disjoint and contained in $B(x,r + \max_{\omega\in F}\diam(\phi_\omega(X_{t(\omega)})))$ to get an upper bound on the sum of their volumes.
\end{remark}

\begin{definition}
\label{definitiongeomirred}
A conformal GDMS $\SS$ on $\R^d$ is said to be \emph{(geometrically) irreducible} if its limit set $J_\SS$ is not contained in any real-analytic submanifold of $\R^d$ of dimension $\leq d - 1$. Otherwise, it is called \emph{(geometrically) reducible}.

\msp\fr If the conformal IFS $\SS$ is irreducible, its limit set $J_\SS$ is also called (geometrically) irreducible. 
\end{definition}

\begin{remark}
The reader should be careful about the distinction between the meaning of the word ``irreducible'' considered here and the one considered in Definition \ref{definitionsymbolirred}. When it is not clear from context which is meant, we use the adverbs ``symbolically'' and ``geometrically'' to clarify. Since IFSes are always symbolically irreducible, the phrase ``irreducible IFS'' always refers to a geometrically irreducible IFS.
\end{remark}

\begin{proposition}
\label{propositionirredequiv}
Let $\SS$ be a symbolically irreducible conformal GDMS.
\begin{itemize}
\item[(i)] If $\SS$ consists entirely of M\"obius transformations (and in particular if $d \geq 3$), then $\SS$ is geometrically reducible if and only if for all $v\in V$, $J_{\SS,v}$ is contained in a generalized sphere.
\item[(ii)] If $d = 2$, then $\SS$ is geometrically reducible if and only if there are conformal embeddings $\psi_v:W_v \to \C \; (v\in V)$ such that $\psi_v(J_{\SS,v}) \subset \R$ for all $v\in V$. 
\item[(iii)] If $\SS$ is geometrically irreducible, then $\SS$ admits a finite geometrically irreducible subsystem.
\end{itemize}
\end{proposition}

\begin{proof}
The backwards directions of (i) and (ii) are both obvious. So suppose that $\SS = \{\phi_e\}_{e\in E}$ is a symbolically irreducible but geometrically reducible conformal GDMS, and fix $\omega\in E_A^\infty$ such that each letter of $E$ occurs infinitely often in $\omega$. This is possible because $\SS$ is assumed to be symbolically irreducible. Let $\pp = \pi(\omega) \in J_\SS$. Then there exist a neighborhood $U \ni \pp$ and a real-analytic map $f:U\to \R$ such that $f'(\pp) \neq \0$ and $J_\SS \cap U \subset f^{-1}(0)$.

Fix $v\in V$, and suppose that $\SS$ consists entirely of M\"obius transformations. Then the family 
\[
(\|\phi_{\omega|n}\|^{-1} f\circ \phi_{\omega|_n})_{n\geq n_0,t(\omega) = v}
\] 
is a normal family, where $n_0$ is chosen large enough so that all these maps are well-defined. Its limits are all of the form $\pi_1\circ \psi$, where $\pi_1:\R^d\to\R$ is projection onto the first coordinate and $\psi$ is a M\"obius transformation. But $J_{\SS,v}$ is contained in the zero sets of these limits, which are generalized spheres. This completes the proof of (i). (If $d \geq 3$, then Liouville's theorem on conformal mappings (see \cite{IwaniecMartin} for a detailed development leading up to the strongest current version including references to the historical background) guarantees that $\SS$ consists entirely of M\"obius transformations.)

On the other hand, suppose that $d = 2$. By shrinking the set $U$ if necessary, we may assume that $f$ is the imaginary part of a conformal embedding $g:U\to\C$. Now choose $n$ large enough so that $\phi_{\omega|_n}(W_v) \subset U$, and let $\psi = g\circ \phi_{\omega|_n}$. We have $J_{\SS,v} \subset \psi^{-1}(\R)$, which completes the proof of (ii).

We now begin the proof of (iii). Let  $\SS = \{\phi_e\}_{e\in E}$ be a symbolically irreducible conformal GDMS, and suppose that all finite subsystems of $\SS$ are geometrically reducible. Since $\SS$ is symbolically irreducible, there exists an increasing sequence $(\SS_n)_1^\infty$ of finite symbolically irreducible subsystems such that $\bigcup_n \SS_n = \SS$. Fix $v\in V$, and note that
\begin{equation}
\label{closurelimit}
J_{\SS,v} \subset \overline{\bigcup_{n \in \N} J_{\SS_n,v}}.
\end{equation}
Now if $d\geq 3$, then each of the sets $J_{\SS_n,v} \; (n \in \N)$ is contained in a generalized sphere. So by \eqref{closurelimit}, $J_{\SS,v}$ is contained in the geometric limit of generalized spheres, and such an object is also a generalized sphere.

If $d = 2$, then for each $n \in \N$ there is a conformal embedding $\psi_n:W_v \to \C$ such that $\psi_n(J_{\SS,v}) \subset \R$. This embedding may be chosen to satisfy $\psi_n(p_v) = 0$ and $|\psi_n'(p_v)| = 1$, where $p_v\in J_{\SS_0,v}$ is a distinguished point. But then by a corollary of Koebe's distortion theorem \cite[Theorem I.1.10]{CarlesonGamelin}, $(\psi_n)_1^\infty$ is a normal family. If $\psi$ denotes one of its limits, then $J_{\SS,v} \subset \psi^{-1}(\R)$.

Finally, if $d = 1$, then $\SS$ is geometrically reducible if and only if $\#(J_{\SS,v}) = 1$ for all $v\in V$, and it is easy to use this to get the desired conclusion.
\end{proof}

\begin{remark}
Note that if $\FF\subset\SS$ is a geometrically irreducible subsystem, then any subsystem $\FF'\subset\SS$ such that $\FF'\supset\FF$ is also geometrically irreducible. So by (iii), if $\SS$ is geometrically irreducible then all sufficiently large finite subsystems of $\SS$ are geometrically irreducible.
\end{remark}

\begin{remark}
We would like to comment on the usefulness of item (ii). It can be used to check irreducibility in some circumstances where it would otherwise be difficult, by using Koebe's distortion theorem \cite[Theorem I.1.6]{CarlesonGamelin}. 
For example, consider a conformal IFS whose limit set contains the vertices of a small triangle. By Koebe's distortion theorem, the images of these points under $\psi_v$ must be almost similar to the original triangle. But triangles in $\R$ are all degenerate, so this shows that only nearly degenerate triangles can be contained in the limit set. So e.g. if the limit set contains the vertices of an equilateral triangle contained in $B(0,1/20)$, where $W = B(0,1)$, then the IFS must be irreducible.\Footnote{Precise calculation: Suppose that $x,y,z\in B(0,1/20)$ constitute an equilateral triangle, say $|y - x| = |z - x| = |z - y| = r$. Then Koebe's distortion theorem shows that
\[
\left|\frac{\psi_v(y) - \psi_v(x)}{\psi_v(z) - \psi_v(x)}\right| \leq \frac{|\psi_v'(x)|\cdot\left(\frac{r}{(1 - r/(19/20))^2}\right)}{|\psi_v'(x)|\cdot\left(\frac{r}{(1 + r/(19/20))^2}\right)} = \left(\frac{19 + 20r}{19 - 20r}\right)^2 \leq \left(\frac{19 + 2}{19 - 2}\right)^2 < 2,
\]
and similarly for the other permutations of $x,y,z$. This shows that $\psi_v(x),\psi_v(y),\psi_v(z)$ cannot form a degenerate triangle, since the ratio between the shortest and longest sides of a degenerate triangle is always at least 2.}
\end{remark}

We are now ready to use Corollary \ref{corollary1} to prove the following result:

\begin{theorem}\label{maintheorem1}
If $\SS = \{\phi_e:X\to X\}_{e\in E}$ is an irreducible conformal IFS on $\R^d$, then
\[
\HD(\BA_d\cap J_\SS)=\HD(J_\SS).
\]
\end{theorem}

\begin{proof}
By Proposition \ref{propositionirredequiv}(iii), there exists a finite set $D\subset E$ such that the finite system $\SS_D\df\{\phi_e\}_{e\in D}$ is irreducible. Since for every set $F\supset D$ we have $J_F\supset J_D$, we conclude that for every $F\supset D$ the system $\SS_F$ is also irreducible. Thus by Corollary \ref{corollary1},
\begin{equation}\label{1_2015_10_16}
\HD(\BA_d\cap J_{\SS_F})=\HD(J_{\SS_F}).
\end{equation}
Now by \cite[Theorem~4.2.13]{MauldinUrbanski2} (cf. \cite[Theorem~4.2.11]{MauldinUrbanski2}), we have $\HD(J_\SS)=\sup\{\HD(J_{\SS_F})\}$, where the supremum is taken over all finite sets $F \subset E$, and by the monotonicity of Hausdorff dimension the same is true if the supremum is restricted to finite sets $F \subset E$ that contain $D$. We infer from \eqref{1_2015_10_16} that
\[
\HD(\BA_d\cap J_\SS)
\ge \sup\{\HD(\BA_d\cap J_{\SS_F})\}
=\sup\{\HD(J_{\SS_F})\}
=\HD(J_\SS),
\]
where again both suprema are taken over all finite sets $F \subset E$ containing $D$. This completes the proof.
\end{proof}

\fr As an immediate consequence of this theorem, we get the following:

\begin{corollary}\label{BA_IFS_3}
If $\SS = \{\phi_e:X\to X\}_{e\in E}$ is a conformal IFS on $\R^d$ such that $\HD(J_\SS)>d-1$, then
\[
\HD(\BA_d\cap J_\SS)=\HD(J_\SS).
\]
\end{corollary}

We proceed to extend the Theorem \ref{maintheorem1} to the realm of GDMSes. If $\SS = \{\phi_e:X\to X\}_{e\in E}$ is a maximal conformal GDMS, then we can canonically associate to it some conformal IFSes. Namely, if $v\in V$ is a vertex, then we can look at the set
\[
E_v^*\df\{\omega\in E_A^*: \text{$t(\omega)=v$, $i(\omega)=v$, and $\omega$ has no proper subword with this property}\}.
\]
Then
\[
\SS_v\df\{\phi_\omega:\omega\in E_v^*\}
\]
is an IFS whose limit set is contained in $J_{\SS,v}$. If the system $\SS$ is geometrically irreducible, then so are all the systems $\SS_v$ ($v\in V$). If in addition $\SS$ is symbolically irreducible, then using the facts that Hausdorff dimension is $\sigma$-stable, increasing with respect to inclusions, and invariant under bi-Lipschitz maps, we get that $\HD\left(J_{\SS_v}\right)=\HD(J_\SS)$ for all $v\in V$. Therefore, as an immediate consequence of Theorem~\ref{maintheorem1}, we get the following:

\begin{corollary}\label{BA_GDMS}
If $\SS = \{\phi_e:X_{t(e)}\to X_{i(e)}\}_{e\in E}$ is a maximal conformal GDMS on $\R^d$ that is geometrically irreducible and symbolically irreducible, then
\[
\HD(\BA_d\cap J_\SS)=\HD(J_\SS).
\]
\end{corollary}

\section{Rational and transcendental functions}
\label{sectionmeromorphic}

\fr We now proceed to derive applications of Corollary \ref{corollary1} to meromorphic dynamics, rational and transcendental alike. We consider the following two cases:
\begin{itemize}
\item[Case 1.] $\DD = \what\C$, and $f:\DD\to\what\C$ is a rational function of degree at least two;
\item[Case 2.] $\DD = \C$, and $f:\DD\to\what\C$ is a transcendental meromorphic function;
\end{itemize}
and commonly refer to them as meromorphic dynamical systems.
We recall that the \emph{Fatou set} of $f$ consists of all those points $z\in\DD$ that admit some neighborhood $U_z\subset\DD$ such that the iterates of $f$ are well-defined on $U_z$ and their restrictions to $U_z$ form a normal family. The \emph{Julia set} $J = J_f$ of $f$ is defined to be the complement of the Fatou set of $f$ in $\what\C$. The Julia set $J_f$ is a nonempty perfect subset of $\what\C$, enjoying the following invariance properties:
\[
f(J_f\cap\DD)\subset J_f \ \text{ and } \ f^{-1}(J_f) = J_f\cap\DD.
\] 
Note that in Case 2, the Julia set contains $\infty$, but in Case 1, this may or may not be true. For the basic properties of Julia sets of rational functions the reader is advised to consult the books \cite{Beardon_book2, CarlesonGamelin, Milnor2, PrzytyckiUrbanski, Steinmetz} and the survey article \cite{Urbanski4}; for the topological dynamics of transcendental functions the survey \cite{Bergweiler1} is a good place to start. We now give several definitions and notations:

\begin{notation}
In the sequel, given a point $z\in \C$ and a radius $r > 0$ we denote by $B_\euc(z,r)\subset\C$ the Euclidean open ball centered at $z$ with radius $r$, while by $B_\sph(z,r)\subset\what\C$ we denote the corresponding ball defined by means of the spherical metric on $\what\C$. We use similar notation for distances, thickenings, diameters, and derivatives taken with respect to the Euclidean or spherical metrics.
\end{notation}

\begin{convention}
The symbols $\lesssim$, $\gtrsim$, and $\asymp$ will denote coarse multiplicative asymptotics. For example, $A\lesssim_K B$ means that there exists a constant $C > 0$ (the \emph{implied constant}), depending only on $K$, such that $A\leq C B$. In general, dependence of the implied constant(s) on universal objects will be omitted from the notation.
\end{convention}

\begin{definition}
\label{definitionradial}
The \emph{radial} (or \emph{conical}) \emph{Julia set} of $f$, denoted $J_\rad(f)$, or just $J_\rad$, consists of all those points $z\in J_f$ for which there exists $\delta > 0$ such that for infinitely many $n\in\N$, the map $f^n$ admits an analytic local inverse branch $f_z^{-n}:B_\sph(f^n(z),\delta)\to\DD$ sending $f^n(z)$ to $z$.
\end{definition}

\begin{remark}
Several different definitions of the ``radial Julia set'' have appeared in the literature. In the context of rational functions, Definition \ref{definitionradial} was the first definition of the radial limit set; it appeared in \cite{Urbanski8} after appearing implicitly in \cite{Lyubich2}. Alternative definitions appeared in \cite{LyubichMinsky,Przytycki2}, but the original definition came to be more popular \cite{DMNU, McMullen_conformal_2, PrzytyckiUrbanski, Urbanski4}. However, by the time Definition \ref{definitionradial} was generalized to the context of transcendental functions \cite{Rempe, BKZ, Zheng}, a different definition for that context had already appeared in \cite{UrbanskiZdunik} and was widely cited \cite{Kotus, CoiculescuSkolruski, MayerUrbanski4, MayerUrbanski3, MayerUrbanski3}. This definition agrees with Definition \ref{definitionradial} in the case of hyperbolic exponential functions,\Footnote{This claim is asserted in \cite[p.1166]{BKZ} and can be proven as follows. If $f$ is a hyperbolic exponential function, then by definition, $f$ has an attracting periodic cycle, and by \cite[Theorem 7]{Bergweiler1}, $0$ is in the attracting basin of this periodic cycle. So the distance between $J_f$ and the forward orbit of $0$ is positive. If $B\subset\C$ is a ball disjoint from the forward orbit of $0$, then for all $n\in\N$, every connected component of $f^{-n}(B)$ is isomorphic to $B$ via $f^n$. So a point $z\in J_f$ is radial if and only if there exists $\epsilon > 0$ such that for infinitely many $n$, $B_\sph(f^n(z),\epsilon)$ is such a ball. This holds if and only if the sequence $(f^n(z))_1^\infty$ does not converge to $\infty$.} but not in general. A third definition in the context of transcendental functions appeared in \cite{UrbanskiZinsmeister, SkorulskiUrbanski}.
\end{remark}

\begin{definition}
A compact set $L \subset \DD$ is called \emph{hyperbolic} (or \emph{expanding}) if $f(L)\subset L$ and there exists $n \in \N$ such that $|(f^n)'(z)|>1$ for all $z\in L$. Note that since $L$ is compact, this implies that $\inf\{|(f^n)'(z)|:z\in L\}>1$.
\end{definition}

\begin{definition}
An \emph{inverse branch IFS} of $f$ is a conformal IFS whose elements are local analytic inverse branches of (positive) iterates of $f$. 
\end{definition}

\begin{definition}\label{definitionfourquantities}
We define four numerical quantities associated with a meromorphic dynamical system:

\msp\begin{enumerate}
\item $\cD(J_f)$ is defined to be the infimum of the exponents of all locally finite conformal measures supported on $J_f$. It is called the \emph{conformal dimension} of $J_f$.

\msp\item $\DynD(J_f)$ is defined to be the supremum of the Hausdorff dimensions of all $f$-invariant ergodic probability measures $\mu$ on $J_f$ with positive Lyapunov exponent\Footnote{The \emph{Lyapunov exponent} of a measure $\mu$ is the number $\chi_\mu \df \int \log|f'| \;\dee\mu$.} that satisfy $\Supp(\mu) \subset \DD$. The quantity $\DynD(J_f)$ is called the \emph{dynamical dimension} of $J_f$.

\item $\hD(J_f)$ is defined to be the supremum of the Hausdorff dimensions of all hyperbolic subsets of $J_f$. It is called the \emph{hyperbolic dimension} of $J_f$.

\msp\item $\hDS(J_f)$ is defined to be the supremum of the Hausdorff dimensions of all limit sets of finite inverse branch IFSes of $f$. We call it the \emph{IFS dimension} of $J_f$.
\end{enumerate}
\end{definition}

\begin{remark}
\label{remarkPUresults}
The measures appearing in the definition of the dynamical dimension, namely the $f$-invariant ergodic measures $\mu$ satisfying $\Supp(\mu) \subset \DD$ (i.e. $\mu$ is supported on a compact subset of $\DD$), are especially transparent from a dynamical perspective. This is because if $\mu$ is such a measure and $X = \Supp(\mu)$, then we have $f\in \AA(X)$ in the notation of \cite[p.295]{PrzytyckiUrbanski}, and so the results of \cite[Chapter 11]{PrzytyckiUrbanski} apply. For example, $\mu$ gives full measure to the radial limit set $J_\rad(f)$, as is seen by combining the first assertion of \cite[Theorem~11.2.3]{PrzytyckiUrbanski} (i.e. existence of inverse branches for almost every point of the Rokhlin extension) with Poincar\'e's recurrence theorem. Moreover, if $\mu$ has positive Lyapunov exponent, then it is exact dimensional of dimension $\hmu(f)/\chi_\mu(f)$, meaning that for $\mu$-a.e. $x\in\DD$, the formula
\begin{equation}
\label{volumelemma}
\liminf_{r\to 0} \frac{\log\mu(B(x,r))}{\log(r)} = \HD(\mu) = \frac{\hmu(f)}{\chi_\mu(f)}
\end{equation}
holds \cite[Theorem~11.4.2]{PrzytyckiUrbanski}. Here $\hmu(f)$ and $\chi_\mu(f)$ denote the entropy and Lyapunov exponent of the measure $\mu$, respectively. The formula \eqref{volumelemma} is known as the \emph{Volume Lemma}. One can assume that $\mu$ has positive entropy rather than positive Lyapunov exponent in the statement of the Volume Lemma, since by Ruelle's inequality \cite[Theorem~11.1.1]{PrzytyckiUrbanski}, the former implies the latter.

In the case of rational functions, the assumption $\Supp(\mu) \subset \DD$ is trivially satisfied (since $\DD = \what\C$), so in that case our definition agrees with the standard definition of dynamical dimension \cite{Urbanski4, Rempe, SUZ1}.

It is natural to ask whether the assumption $\Supp(\mu) \subset \DD$ can be replaced by some integrability condition on the measure $\mu$, but we do not address this question here.
\end{remark}

\ignore{
\begin{remark}
In the case of rational functions, the dynamical dimension is usually defined as the supremum of the Hausdorff dimensions of the ergodic $f$-invariant measures of positive entropy \cite{Urbanski4, Rempe, SUZ1}. To see that this definition agrees with ours in this special case, we note that if $f$ is a rational function, then
\begin{itemize}
\item all $f$-invariant measures have finite Lyapunov exponent, since the Lyapunov exponent can never be more than the finite number $\log\sup |f'|$;
\item all $f$-invariant measures of positive entropy also have positive Lyapunov exponent, by Ruelle's inequality \cite[Theorem 11.1.1]{PrzytyckiUrbanski};\Footnote{Here we use Ruelle's inequality for rational functions. It is not clear whether Ruelle's inequality holds for all transcendental functions.}
\item those $f$-invariant measures with positive Lyapunov exponent but zero entropy do not contribute to the supremum, by the Volume Lemma (Remark \ref{remarkPUresults}).
\end{itemize}
In the case of transcendental functions, our definition differs from the naive extension of the standard one to the transcendental case. It is evident from the proof of our Theorem \ref{All_Dim_=} below that our definition is the more natural one in this setting, and it is not clear whether the other definition is equivalent.
\end{remark}
}

\begin{remark}
We make a few historical remarks about these definitions. The first to appear was the \emph{conformal dimension}, which was never given a name but in \cite{DenkerUrbanski3}, it was shown to be equal to the Hausdorff dimension of the Julia set in the case of expansive rational functions. Next was the \emph{dynamical dimension}, which was defined in \cite{DenkerUrbanski2}, where it was shown to be equal to the conformal dimension as well as to the zero of a certain pressure function, this time in the case of arbitrary rational functions. Next was the \emph{hyperbolic dimension}, which in the context of rational functions was introduced independently in \cite{Shishikura} and \cite{PrzytyckiUrbanski} (the first draft of the latter had been circulating since the mid '90s), and was proven to equal the dynamical dimension in \cite[Theorem 12.3.11]{PrzytyckiUrbanski}. After that, McMullen \cite{McMullen_conformal_2} showed that the Hausdorff dimension of the radial Julia set of a rational function is equal to its hyperbolic dimension, and Rempe \cite{Rempe} extended this result to transcendental functions. Finally, the \emph{IFS dimension} has not been defined explicitly before now, but the proof of Rempe's theorem shows that it is equal to the hyperbolic dimension and the dimension of the radial Julia set. The following theorem extends Rempe's result by showing that the equality $\DynD(J_f) = \hD(J_f)$ and the inequality $\hD(J_f) \leq \cD(J_f)$ can be extended to the transcendental setting as well:
\end{remark}


\begin{theorem}\label{All_Dim_=}
If $f:\DD\to\what\C$ is a meromorphic dynamical system, then
\begin{equation}
\label{dimensions}
\DynD(J_f) = \hD(J_f) = \HD(J_\rad(f)) = \hDS(J_f) \leq \cD(J_f).
\end{equation}
\end{theorem}
\fr The common number appearing in \eqref{dimensions} will be denoted by $\delta_f$ and called the \emph{dynamically accessible dimension} of $J_f$.
\begin{proof}
As stated earlier, the proof of \cite[Theorem~3.1]{Rempe} shows that
\begin{equation}
\label{rempe}
\hD(J_f) = \HD(J_\rad(f)) = \hDS(J_f),
\end{equation}
although this theorem only explicitly states that $\hD(J_f) = \HD(J_\rad(f))$. Next, the inequality
\begin{equation}\label{DynLeqJrad}
\DynD(J_f)\le \HD(J_\rad(f))
\end{equation}
follows from the fact that any $f$-invariant ergodic probability measure $\mu$ on $J_f$ of positive Lyapunov exponent that satisfies $\Supp(\mu) \subset \DD$ gives full measure to the radial limit set (see Remark \ref{remarkPUresults}). To complete the proof, we need to prove the inequalities
\begin{align}\label{HypLeqDyn}
\hD(J_f) &\leq \DynD(J_f),\\ \label{HypLeqConf}
\hD(J_f) &\leq \cD(J_f).
\end{align}

Let $L \subset \DD$ be a hyperbolic set, and for each $t\in\R$ let $\mathrm P(t)$ denote the topological pressure of the dynamical system $f|_L:L\to L$ with respect to the potential function $-t\log|f'|$. Choose $n$ large enough so that $\inf_L |(f^n)'| > 1$. For all $t,u \geq 0$, we have
\begin{equation}
\label{pressurebounds}
\mathrm P(t) - u\log\sup_L |f'| \leq \mathrm P(t + u) \leq \mathrm P(t) - u\frac{1}{n}\log\inf_L |(f^n)'|.
\end{equation}
In particular, $\mathrm P$ is continuous.

Since $L\subset\DD$ is compact, the map $f|_L:L\to L$ is of finite degree, so we have $\mathrm P(0) = \htop(f|_L) < \infty$. Combining with \eqref{pressurebounds} gives $\lim_{t\to+\infty} \mathrm P(t) = -\infty$. On the other hand, $\mathrm P(0) = \htop(f|_L) \geq 0$, so by the intermediate value theorem, $\mathrm P$ has a unique zero in the interval $\CO 0\infty$, which we will denote by $h$. A standard argument in complex dynamics (cover $L$ by finitely many small balls and then look at their holomorphic univalent pull-backs to $L$ of the same order; their diameters are comparable, by bounded distortion, to moduli of derivatives) tells us that 
\[
\HD(L)\le h. 
\]
Denote by $\mu$ an ergodic equilibrium state of the dynamical system $f|_L:L\to L$ with respect to the potential function $-h\log|f'|$. Such an equilbrium state exists because the system is expansive \cite[Theorem~3.5.6]{PrzytyckiUrbanski}. By the definition of an equilibrium state, we have
\[
\h_{\mu}-h\chi_{\mu} = \mathrm P(h)=0,
\]
where $\chi_{\mu}$ denotes the Lyapunov exponent of $\mu$, i.e. $\chi_{\mu} = \int\log|f'|\;\dee\mu$. Since $f|_L:L\to L$ is expanding, we have $\chi_{\mu} > 0$. Therefore, by the Volume Lemma (see Remark \ref{remarkPUresults}), we have
\[
\HD(L)\le h = \frac{\hmu}{\chi_{\mu}}=\HD(\mu)\le \DynD(J_f).
\]
Taking the supremum over all such sets $L$ thus proves \eqref{HypLeqDyn}.

Again let $L \subset \DD$ be a hyperbolic set, and let $\mu$ be a locally finite $s$-conformal measure on $J_f$. Let $n\in\N$ and $\epsilon > 0$ be chosen so that $|(f^n)'(z)| > 1$ for all $z\in \NN(\LL,\epsilon)$. Fix $x\in L$ and $r > 0$, and let $m\in\N$ be the largest number such that $\diam(f^{mn}(B(x,r))) \leq \epsilon/2$. By Koebe's distortion theorem, $f^{mn}$ has bounded distortion on $B(x,r)$ and thus
\[
\mu\big(B(x,r)\big) \asymp r^s \mu\big(f^{mn}(B(x,r))\big) \asymp r^s.
\]
Since $x$ was arbitrary, standard results (e.g. \cite[Theorem 8.2]{MSU}) imply that $\HD(L) \leq s$; taking the supremum over $L$ and the infimum over $\mu$ gives \eqref{HypLeqConf}. Combining \eqref{rempe}, \eqref{DynLeqJrad}, \eqref{HypLeqDyn}, and \eqref{HypLeqConf} completes the proof.
\end{proof}

\begin{remark}
\label{remarkequalityholds}
In the context of rational functions $f$ it is known that equality holds in \eqref{dimensions} (see \cite[Theorem 2.3]{DenkerUrbanski2} and \cite[p.310, para. 2]{Przytycki3}), but in the realm of transcendental functions the situation is more subtle, and it is not clear whether such a theorem is true. One case where equality is known to hold is the class of \emph{dynamically regular functions of divergence type} that were introduced and considered in detail in \cite{MayerUrbanski3, MayerUrbanski2}. The fact that equality holds in this case follows from combining various theorems in \cite{MayerUrbanski2}, as we now show:
\begin{itemize}
\item[1.] \cite[Theorem 8.3]{MayerUrbanski2} says that if $f$ is dynamically regular of divergence type, then there exists $\tau_0 < \underline\alpha_2$ such that for all $\tau_0 < \tau < \underline\alpha_2$, the pressure of the potential function $\phi(z) = -\delta_f \log|f'|_\tau$ is zero, where the notation is as in \cite[(4.6)]{MayerUrbanski2}.
\item[2.] \cite[Theorem 5.15(2)]{MayerUrbanski2} says that if $f$ is dynamically semi-regular and $\phi$ is tame, then there exists a unique $e^{P(\phi) - \phi}$-conformal probability measure. Note that the definition of ``conformal'' in \cite[Definition 5.10]{MayerUrbanski2} is different from our definition.
\item[3.] Since $P(\phi) = 0$ as shown in step 1, this means the measure constructed in step 2 is $e^{-\phi}$-conformal, which in our terminology means that it is $\delta_f$-conformal with respect to the metric $d\tau$. The measure can be converted into a new measure conformal with respect to the Euclidean metric by multiplying by the factor $z\mapsto |z|^{\alpha_2 \delta_f}$; cf. \cite[p.26]{MayerUrbanski2}. The resulting measure is locally finite (on $\DD = \C$) but not necessarily finite.
\item[4.] If we can show that the function $\phi$ from step 1 is tame, then steps 1-3 show that if $f$ is a dynamically regular function of divergence type, then $f$ has a locally finite $\delta_f$-conformal measure (with respect to the Euclidean metric). Comparing with the definition of tameness \cite[Definition 5.1]{MayerUrbanski2}, we set $h\equiv 0$ and $t = \delta_f$, and we use the second displayed equation on \cite[p.23]{MayerUrbanski2} rather than the first displayed equation, since the second equation is what is actually used in the proofs.\Footnote{In general one should be careful about ``cohomologous'' changes of potential such as the one on \cite[p.23]{MayerUrbanski2}, since these sometimes change the conformal measure by a factor which is not locally bounded. We don't have to deal with this issue because the arguments of \cite{MayerUrbanski2} don't actually use any cohomology between the functions $\phi$ and $\Phi$ in the second displayed equation of \cite[p.23]{MayerUrbanski2}, they just work directly with $\phi$, which is the function that we can show is tame.} With these settings, $\phi$ is tame if and only if $\delta_f > \rho/(\alpha_1 + \underline\alpha_2)$. To prove that this is the case, we show that $\delta_f > \rho/(\alpha_1 + \tau) > \rho/(\alpha_1 + \underline\alpha_2)$. Now the first inequality here follows from setting $q = 0$ in \cite[Lemma 8.2]{MayerUrbanski2} (observing that $T(0) = \delta_f$ by \cite[Theorem 8.3]{MayerUrbanski2}), and the second inequality follows from rearranging the inequality $\tau < \underline\alpha_2$.
\end{itemize}
\end{remark}

\begin{definition}
We say that a meromorphic function $f:\DD\to\what\C$ is \emph{reducible} if some relatively open subset of $J_f$ is contained in a real-analytic curve. We say that $f$ is \emph{irreducible} if it is not reducible.
\end{definition}

\begin{remark}
\label{remarkBEvS}
It was shown by Bergweiler, Eremenko, and van Strien \cite{BergweilerEremenko, EremenkoVanstrien} that a rational function is irreducible if and only if its limit set is not contained in any generalized circle (i.e. either a circle or the union of $\{\infty\}$ with a line). It is not clear whether the same holds for transcendental functions.
\end{remark}


As in the case of infinite IFSes, we prove the full dimension of $\BA_2$ in the radial Julia set of a meromorphic dynamical system by constructing appropriate irreducible large-dimensional subsets:

\begin{proposition}
\label{propositiondimirredIFS}
If $f$ is irreducible, then
\[
\delta_f = \sup\{\HD(J_\SS)\},
\]
where the supremum is taken over all finite irreducible inverse branch IFSes of $f$.
\end{proposition}
\begin{proof}
We show that by modifying slightly the proof of \cite[Theorem 3.1]{Rempe}, the inverse branch IFSes appearing in that proof can be made to be irreducible. It suffices to consider the limit of these IFSes as $N\to\infty$, i.e. the infinite IFS that consists of all the maps $g_j$ ($j\in\N$) on \cite[p.1418]{Rempe}, since by Proposition~\ref{propositionirredequiv}(iii), if this IFS is irreducible then its sufficiently large finite subsystems are as well. Denote this IFS by $\SS_\infty$, and note that $J_\rad(D)\cap D \subset \NN_\sph(J_{\SS_\infty},\delta)$, where the set $J_\rad(D)$ and the parameter $\delta > 0$ are as in \cite[p.1417]{Rempe}.

\begin{claim}
\label{claimJrD}
$\cl{J_\rad(D)} = J_f$.
\end{claim}
\begin{subproof}
The set $D$ was chosen so that $\HD(J_\rad(D)) > d'$, where $d' > 0$ is as in \cite[p.1417]{Rempe}. In particular, $J_\rad(D)$ is uncountable. On the other hand, as observed in \cite[p.1417]{Rempe}, $J_\rad(D)$ can be written as the union of a countable set and a backwards invariant set. Thus, there is some non-exceptional point $z\in J_f$ whose backwards orbit is contained in $J_\rad(D)$. (We recall that a point is \emph{exceptional} if its backwards orbit is finite. By Picard's theorem, there are at most two exceptional points.) By \cite[Lemma 4]{Bergweiler1}, the backwards orbit of $z$ is dense in $J_f$. This completes the proof.
\end{subproof}

It follows from Claim \ref{claimJrD} that $J_f \cap D \subset \cl{J_\rad(f)\cap D} \subset \NN_\sph(J_{\SS_\infty},\delta)$. Since the IFS $\SS_\infty$ depends on the parameter $\delta > 0$ appearing in \cite[p.1417]{Rempe}, from now on we will write $\SS_\infty = \SS_\infty(\delta)$. By contradiction, suppose that for all sufficiently small $\delta > 0$, the IFS $\SS_\infty = \SS_\infty(\delta)$ is reducible. Then by Proposition~\ref{propositionirredequiv}(ii), for each such $\delta$ there is a conformal embedding $\psi = \psi_\delta:D\to\C$ such that $J_{\SS_\infty} \subset \psi_\delta^{-1}(\R)$. So
\[
J_f \cap D \subset \NN_\sph(\psi_\delta^{-1}(\R),\delta).
\]
As before, we can without loss of generality assume that $\psi_\delta(p) = 0$ and $|\psi_\delta'(p)| = 1$, where $p\in D$ is a distinguished point. Note that since $\delta$ is chosen after $D$ in the proof of \cite[Theorem 3.1]{Rempe}, the maps $(\psi_\delta)_\delta$ all have the same domain. Thus $(\psi_\delta)_\delta$ is a normal family, and if $\psi$ denotes one of its limits as $\delta \to 0$, then $J_f\cap D \subset \psi^{-1}(\R)$. Thus $f$ is reducible, a contradiction.
\end{proof}

\fr As an immediate consequence of Corollary \ref{corollary1} and Proposition~\ref{propositiondimirredIFS} (see also Theorem~\ref{All_Dim_=}), we get the following:  

\begin{theorem}\label{maintheorem2}
If $f:\DD\to\what\C$ is an irreducible meromorphic dynamical system, then $\BA_2$ has full dimension in $J_\rad(f)$:
\begin{equation}
\label{MT2equation}
\HD(\BA_2\cap J_f) \ge \HD(\BA_2\cap J_\rad(f)) = \HD(J_\rad(f)) = \delta_f.
\end{equation}
\end{theorem}

Note in particular that Theorem \ref{maintheorem2} applies to all those functions whose dynamically accessible dimension is strictly greater than 1, since such functions are irreducible. For example, by \cite[Theorems 2.1 and 4.5]{UrbanskiZdunik},  we have $\HD(J_\rad(f))>1$ for every function $f$ of the form $f(z) = \lambda e^z$, $\lambda\in\C\setminus\{0\}$. Combining with Theorem \ref{maintheorem2} yields:

\begin{corollary}\label{c4_2015_10_17}
If $\lambda\in\C\setminus\{0\}$ and $f:\C\to\what\C$ is the entire function given by the formula $f(z)=\lambda e^z$, then \eqref{MT2equation} holds.
\end{corollary}

\fr For each nonconstant elliptic (i.e. doubly periodic) function $f:\C\to\what\C$, let $q_f\ge 1$ be the maximum of the orders of the poles of $f$. It has been proved in \cite{KotusUrbanski3} that $\HD(J_f)>\frac{2q_f}{q_f+1}\ge 1$, and in fact the proof showed that
$\HD(J_\rad(f))>\frac{2q_f}{q_f+1}\ge 1$. Moreover, it has been proved in \cite{KotusUrbanski4} that $\delta_f=\HD(J_f)$ for each non-recurrent elliptic function $f:\C\to\what\C$. Therefore, as an immediate consequence of Theorem~\ref{maintheorem2}, we get the following:

\begin{corollary}\label{c5_2015_10_17}
If $f:\C\to\what\C$ is a nonconstant elliptic function, then 
\[
\HD(\BA_2\cap J_f)\ge \HD(\BA_2\cap J_\rad(f)) = \delta_f > \frac{2q_f}{q_f+1}\ge 1.
\]
If in addition $f$ is non-recurrent, then
\[
\HD(\BA_2\cap J_f) = \HD(J_f) = \delta_f.
\]
\end{corollary}

\fr It was proven in \cite{PrzytyckiRivera} that for any rational function $f:\what\C\to\what\C$ that is a topological Collet--Eckmann map (one of a variety of equivalent definitions of such maps is that the diameters of connected components of inverse images of small balls converge to zero exponentially fast\Footnote{Further characterizations of topological Collet--Eckmann maps and their fundamental properties can be found in \cite{PrzytyckiRivera, PrzytyckiRivera2} and the references therein.}), we have $\delta_f=\HD(J_f)$. On the other hand, it was proven in \cite[Proposition 6.1]{Urbanski6} that for any rational function which is non-recurrent (i.e. there are no recurrent critical points in its Julia set), the set $J_f \butnot J_\rad(f)$ is countable, and in particular $\delta_f=\HD(J_f)$. We therefore obtain the following:

\begin{corollary}\label{maincorollary}
If $f:\what\C\to\what\C$ is an irreducible rational function which is either topological Collet--Eckmann or non-recurrent, then 
\[
\HD(\BA_2\cap J_f)=\HD(J_f) = \delta_f.
\]
\end{corollary}

\begin{remark*}
The class of non-recurrent rational maps contains the classes of semi-hyperbolic, sub-hyperbolic, and expansive (i.e. hyperbolic or parabolic) rational functions. This fact follows from well-known equivalent formulations of the definitions of these classes; see \cite[Theorem 2.1]{CJY}, \cite[Theorem V.3.1]{CarlesonGamelin},  and \cite[Theorem 3.1]{Urbanski4}, respectively. Therefore, Corollary \ref{maincorollary} applies to all of those classes.
\end{remark*}

\ignore{
\section{New section}

\begin{theorem}
Let $f:\DD\to\what\C$ be a meromorphic dynamical system, and for each $z\in\DD$ let
\[
\epsilon(z) = \max\big\{\epsilon : f \text{ is injective on $B = B_\sph(z,\epsilon)$ and $\inf_B |f'|_\sph \geq |f'(z)|_\sph/2$}\big\}.
\]
Let $\mu$ be an ergodic $f$-invariant measure on $J_f$ with positive Lyapunov exponent such that
\[
\int \log\epsilon\;\dee\mu > -\infty.
\]
Then $\mu$ gives full measure to $J_\rad$. Moreover, $\mu$ is exact dimensional of dimension $\hmu/\chi_\mu$.
\end{theorem}
\begin{proof}
Fix $\lambda \in (0,1)$ to be determined, close to 1. Let $(z_n)_{n\in\N}$ be chosen randomly with respect to the Rokhlin extension of $\mu$, so that $f(z_{n + 1}) = z_n$ for all $n$. [Need to explain what this means\internal] By the Borel--Cantelli lemma, the integrability condition implies that we can assume that
\begin{equation}
\label{epsilonbound}
\epsilon(z_n) \geq \lambda^n
\end{equation}
for all but finitely many $n$. Next, for each $\delta > 0$ and $z\in \what\C$ let
\begin{align*}
A_\delta(z) &= \min_{B_\sph(z,\delta)} |f'|_\sph,&
B_\delta(z) &= \max(A_\delta(z),|f'(z)|_\sph/2).
\end{align*}
Then by the Lebesgue dominated convergence theorem,
\[
\int \log B_\delta(z) \;\dee\mu(z) \tendsto{\delta\to 0} \chi_\mu > 0,
\]
so by choosing $\delta > 0$ small enough, the ergodic theorem guarantees that
\[
\frac{1}{n}\sum_{i = 1}^n \log B_\delta(z_i) \tendsto{n\to\infty} \int \log B_\delta(z) \;\dee\mu(z) > 0.
\]
So if $\lambda$ is chosen large enough, then for all but finitely many $n$,
\begin{equation}
\label{ergodicbound}
\sum_{i = 1}^n \log B_\delta(z_i) \geq n\log(\lambda^{-1}).
\end{equation}
Let $N$ be large enough so that \eqref{epsilonbound} and \eqref{ergodicbound} hold for all $n\geq N$, and so that $\lambda^N \leq \delta$. For each $n\geq N$, let $\rho_n = 1/\prod_{i = 1}^n B_\delta(z_i)$ and $B_n = B_\sph(z_n,\rho_n)$. We have $\rho_n \leq \lambda^n \leq \min(\delta,\epsilon(z_n))$, and thus $f$ is injective on $B_n$ and $A_{\rho_n}(z_n) \geq B_\delta(z_n)$. This implies that $f(B_n) \supset B_{n - 1}$, so an induction argument shows that for all $n\geq N$, $f^{n - N}$ has an injective inverse branch from $B_N$ to $B_n$. So (assuming that $z_N$ is not a critical point of $f^N$) there exists a neighborhood $U\ni z_0$ such that for all $n$, $f^n$ has an injective inverse branch from $U$ to a domain containing $z_n$. Some standard quantifier arguments finish the proof. [Need to explain more\internal]
\end{proof}

[Missing: Proof of Volume Lemma in this setting\internal]
}

\section{Open questions}

We conclude with a list of open questions:

\begin{question}[Cf. Theorem \ref{All_Dim_=} and Remark \ref{remarkequalityholds}]
Does equality hold in \eqref{dimensions} for all transcendental functions $f$?
\end{question}

\begin{question}[Cf. Remark \ref{remarkBEvS}]
Does there exist a transcendental function whose limit set is contained in a curve, but not in a generalized circle?
\end{question}

\appendix
\section{Semi-hyperbolic rational functions, hyperplane diffuseness, and absolute decay}
\label{appendixsemihyperbolic}

In Theorem \ref{maintheorem2} (resp. Corollary \ref{maincorollary}), we proved that the set of badly approximable vectors has full dimension in the radial Julia set (resp. total Julia set) of any rational function (resp. any topological Collet--Eckmann or non-recurrent rational function). To some extent, this theorem make results like Theorem \ref{theorem3} obsolete: a main motivation for Theorem \ref{theorem3} was to prove the full dimension of $\BA_2$ for hyperbolic rational functions via Theorem \ref{theoremkleinbockweiss}, but we now have another way to prove this using much weaker assumptions than hyperbolicity. A similar situation exists regarding the other main motivation for Theorem \ref{theorem3}, namely extremality; cf. \cite[Theorem 1.18]{DFSU_GE2} and the surrounding discussion. Nevertheless, Theorem \ref{theorem3} is still interesting from a geometrical point of view. In this appendix we prove that Theorem \ref{theorem3} can be generalized to a class of functions which is much broader than the class of hyperbolic rational functions, while still being smaller than the classes of non-recurrent and topological Collet--Eckman rational functions. This is the class of semi-hyperbolic rational functions:

\begin{definition}[{\cite[p.5]{CJY}}]
\label{definitionsemihyp}
A rational function $f:\what\C\to\what\C$ is called \emph{semi-hyperbolic} if there exist $\epsilon > 0$ and $D\in\N$ such that for all $x\in J_f$ and $n\in\N$, the topological degree of $f^n\given f^{-n}(B_\sph(x,\epsilon))_x$ is at most $D$. Here $U_x$ denotes the connected component of $U$ that contains $x$.
\end{definition}

\begin{theorem}
\label{theoremsemihyperbolic}
Let $f:\what\C\to\what\C$ be an irreducible semi-hyperbolic rational function, and let $\delta$ be the dynamically accessible dimension of $f$. Then the $\delta$-dimensional Hausdorff measure restricted to $J_f$ is absolutely friendly and Ahlfors regular of dimension $\delta$.
\end{theorem}

We will prove Theorem \ref{theoremsemihyperbolic} via lemmas. In the process, we will provide a new technique for proving that measures are absolutely friendly. This technique can be used to answer a question raised in \cite{BFKRW} (see Remark \ref{remarkBFKRW} below).

\begin{lemma}
\label{lemmaBD}
Let $U_1,U_2$ be simply connected domains and let $f:U_1\to U_2$ be a proper holomorphic map of degree $D$ (i.e. each point of $U_2$ has exactly $D$ preimages counting multiplicity). Fix $C_1 > 0$. Then there exists $C_2 > 0$, depending only on $D$ and $C_1$, such that if $K_2\subset U_2$ is a compact set with hyperbolic diameter $\leq C_1$ (with respect to the hyperbolic metric of $U_2$), and $K_1\subset U_1$ is a connected component of $f^{-1}(K_2)$, then
\begin{itemize}
\item[(i)] the hyperbolic diameter of $K_1$ with respect to the hyperbolic metric of $U_1$ is $\leq C_2$; and
\item[(ii)] there exists a constant $\lambda > 0$ such that for all $x\in K_1$,
\begin{equation}
\label{BD}
|f'(x)| \asymp_{C_2} \lambda\prod_{i = 1}^{D - 1} |x - a_i|,
\end{equation}
where $(a_i)_1^{D - 1}$ is a list of the critical points of $f$, counting multiplicity.
\end{itemize}
\end{lemma}
\begin{proof}
By the Riemann mapping theorem, it suffices to prove each part for a single pair of sets $(U_1,U_2)$. (To prove that this reduction is valid for part (ii), use part (i) and Koebe's distortion theorem.) We prove (i) with $U_1 = U_2 = \B$ (the unit ball) and (ii) with $U_1 = U_2 = \H$ (the upper half-plane).

Let $U_1 = U_2 = \B$, and without loss of generality suppose that $0\in K_2$. Then $K_2 \subset B(0,C_{1.1})$, where $C_{1.1} < 1$ is chosen so that $\dist_\B(0,C_{1.1}) = C_1$. Here and from now on $\dist_\B$ denotes distance taken with respect to the hyperbolic metric of $\B$. Since $f:\B\to\B$ is proper of degree $D$, we can write $f$ as a Blaschke product
\[
f(x) = \prod_{i = 1}^D \frac{x - b_i}{1 - b_i} \frac{1 - \wbar b_i^{-1}}{x - \wbar b_i^{-1}},
\]
where $b_1,\ldots,b_D \in B(0,1)$ are the zeros of $f$ counting multiplicity, and the second factor is omitted if $b_i = 0$. If $f(x) \in K_2$, then there exists $i = 1,\ldots,D$ such that
\[
\left|\frac{x - b_i}{1 - b_i} \frac{1 - \wbar b_i^{-1}}{x - \wbar b_i^{-1}}\right| \leq C_{1.1}^{1/D},
\]
in which case the hyperbolic distance from $x$ to $b_i$ is at most $C_{1.2} \df \dist_\B(0,C_{1.1}^{1/D})$. So $K_1 \subset \bigcup_{i = 1}^D B_\B(b_i,C_{1.2})$ and thus since $K_1$ is connected, $\diam_\B(K_1) \leq C_2 \df 2D C_{1.2}$, which completes the proof of (i).

Now let $U_1 = U_2 = \H$, and without loss of generality suppose that $f(\infty) = \infty$. Then $f$ is a real polynomial of degree $D$, and \eqref{BD} holds with equality (where $\lambda$ is the leading coefficient of $f$). This completes the proof of (ii).
\end{proof}

\begin{definition}[{\cite[Definition 4.2]{BFKRW}}]
\label{definitionhyperplanediffuse}
A closed set $J\subset\R^d$ is \emph{hyperplane diffuse} if there exists $\gamma > 0$ such that for all $\xx\in J$ and $0 < r \leq 1$, and for every affine hyperplane $\LL\subset\R^d$, we have
\[
J\cap B(\xx,r) \butnot \NN(\LL,\gamma r) \neq \emptyset.
\]
\end{definition}

\begin{lemma}
\label{lemmahyperplanediffuse}
Let $f:\what\C\to\what\C$ be an irreducible semi-hyperbolic rational function. Then $J_f$ is hyperplane diffuse.
\end{lemma}
\begin{proof}
Suppose that $f$ is semi-hyperbolic but $J_f$ is not hyperplane diffuse, and we will show that $f$ is reducible. For each $k\in\N$, let $p = p_k\in\C$, $r = r_k \in \OC 01$, and an affine hyperplane $\LL = \LL_k \subset \C$ be chosen so that $J_f\cap B_\euc(p_k,r_k) \subset \NN_\euc(\LL_k,r_k/k)$. Let $B = B_k = B_\euc(p_k,r_k)$, and let $n = n_k$ be the largest number such that $\diam_\sph(f^n(B)) \leq \epsilon/2$, where $\epsilon > 0$ is as in Definition \ref{definitionsemihyp}. Let $K_2 = B_\sph(f^n(p),\epsilon/2)$, $U_2 = B_\sph(f^n(p),\epsilon)$, $K_1 = f^{-n}(K_2)_p$, and $U_1 = f^{-n}(U_2)_p$. By Definition \ref{definitionsemihyp}, the degree $D_k$ of the proper holomorphic map $f^n:U_1\to U_2$ is $\leq D$, so by Lemma \ref{lemmaBD}, there exists $\lambda = \lambda_k > 0$ such that for all $x\in B \subset K_1$,
\[
|(f^n)'(x)|_{\euc\to\sph} \asymp \lambda \prod_{i = 1}^{D_k - 1} |x - a_i|,
\]
where $a_1,\ldots,a_{D_k - 1}$ are the critical points of $f^n\given U_1$. Here the notation $|\cdot|_{\euc\to\sph}$ means that the derivative is taken with respect to the Euclidean metric in the domain but the spherical metric in the codomain.

Now let the isomorphism $g = g_k : \B \df B(0,1)\to B$ be given by the formula $g(x) = p + rx$, and for each $i$ let $b_i = g^{-1}(a_i)$. Then for all $x\in \B$,
\[
|(f^n\circ g)'(x)|_{\euc\to\sph} \asymp \lambda r^{D_k} \prod_{i = 1}^{D_k - 1} |x - b_i|.
\]
Let $h_k = f^{n_k}\circ g_k$. Since $\diam_\sph(h_k(\B)) \leq \epsilon/2$ for all $k$, $(h_k)_1^\infty$ is a normal family. Consider a convergent subsequence $h_k \dashrightarrow h:\B\to\what\C$. We claim that $h$ is non-constant. Indeed, since
\[
1 \asymp \diam_\sph(h_k(\B)) \lesssim \max_\B |h_k'|_{\euc\to\sph} \asymp \lambda_k r_k^{D_k},
\]
if $x_1,\ldots,x_D\in \B$ are distinct points then
\[
1 \lesssim \lambda_k r_k^{D_k} \asymp \max_{i = 1}^D |h_k'(x_i)|_{\euc\to\sph} \dashrightarrow \max_{i = 1}^D |h'(x_i)|_{\euc\to\sph},
\]
so for some $i = 1,\ldots,D$ we have $|h'(x_i)|_{\euc\to\sph} > 0$.

For each $k$ let $\w\LL_k = g_k^{-1}(\LL_k)$. Then $J_f\cap h_k(\B) \subset h_k(\NN_\euc(\w\LL_k,1/k))$. Taking the limit along a subsequence as $k\to\infty$, we get $J_f \cap h(\B) \subset h(\w\LL)$, where $\w\LL_k \dashrightarrow \w\LL$. Since $J_f \ni f^{n_k}(p_k) = h_k(0) \dashrightarrow h(0)$, we have $J_f\cap h(\B) \neq \emptyset$ and thus we have constructed a nonempty relatively open subset of $J_f$ contained in the real-analytic image of a line segment. Since such a set can be written as the union of finitely many real-analytic curves, this show that $f$ is reducible, completing the proof.
\end{proof}

\begin{lemma}
\label{lemmaHDFimpliesAD}
Any doubling measure whose topological support is hyperplane diffuse is absolutely decaying.
\end{lemma}
\begin{remark}
\label{remarkBFKRW}
Incidentally, Lemma \ref{lemmaHDFimpliesAD} resolves a question of Broderick, Kleinbock, Reich, Weiss, and the second-named author \cite[p.14]{BFKRW}: ``Let $K$ be a hyperplane diffuse subset of $\R^d$. Whether or not it is possible to construct an absolutely decaying measure $\mu$ with $\Supp\mu = K$ is an open question.'' Indeed, since subsets of $\R^d$ are automatically doubling as metric spaces, the main result of \cite{LuukkainenSaksman} implies that they support doubling measures, and Lemma \ref{lemmaHDFimpliesAD} implies that these measures are absolutely decaying if the set in question is hyperplane diffuse.
\end{remark}
\begin{proof}[Proof of Lemma \ref{lemmaHDFimpliesAD}]
Indeed, let $\mu$ be such a measure, and let $J$ be the topological support of $\mu$. Let $\gamma > 0$ be as in Definition \ref{definitionhyperplanediffuse}.
\begin{claim}
Let $\pp\in J$ be any point, and let $\LL$ be any affine hyperplane. Fix $0 < r < R < 1$, and let
\begin{align*}
S_1 &= B(\pp,R)\cap \thickvar{\LL}{\gamma r/2}\\
S_2 &= B(\pp,R + r)\cap \thickvar{\LL}{(\gamma + 2) r}
\end{align*}
Then
\begin{equation}
\label{epsilon}
\mu(S_1) \leq (1 - \varepsilon)\mu(S_2)
\end{equation}
for some $\varepsilon > 0$ which does not depend on $\pp$, $\LL$, or $r$.
\end{claim}
\begin{subproof}
Let $(\xx_i)_1^N$ be a maximal $4r$-separated sequence in $S_1\cap J$. Fix $i = 1,\ldots,N$. Since $J$ is hyperplane diffuse, we have
\[
J\cap B(\xx_i,r)\setminus\thickvar{\LL}{\gamma r}\neq \emptyset;
\]
let $\yy_i$ be any member of this set. We observe that
\[
B(\yy_i,\gamma r/2) \subseteq B(\xx_i,2r)\subseteq S_2
\]
and on the other hand
\[
B(\yy_i,\gamma r/2)\cap S_1 = \emptyset
\]
and so, since the balls $B(\xx_i,2r)$ ($i = 1,\ldots,N$) are disjoint, we have
\[
\mu(S_2\butnot S_1) \geq \sum_{i = 1}^N \mu(B(\yy_i,\gamma r/2))
\asymp \sum_{i = 1}^N \mu(B(\xx_i,4r)) \geq \mu(S_1).
\]
(The asymptotic holds since $\mu$ is doubling.) Rearranging yields (\ref{epsilon}).
\end{subproof}
Iterating $n$ times yields the following corollary:
\begin{corollary}
Let $\pp\in J$ be any point, and let $\LL$ be any affine hyperplane. Fix $0 < r < R < 1$ and $n\in\N$, and let
\begin{align*}
S_1 &= B(\pp,R)\cap \thickvar{\LL}{r}\\
S_2 &= B\left(\pp,R + \frac{2}{\gamma}\sum_{j = 0}^{n - 1}K^j r\right)\cap \thickvar{\LL}{K^n r},
\end{align*}
where $K = 2(\gamma + 2)/\gamma > 1$. Then
\begin{equation}
\label{epsilon2}
\mu(S_1) \leq (1 - \varepsilon)^n\mu(S_2).
\end{equation}
\end{corollary}
To show that $\mu$ is absolutely decaying, let $B = B(\pp,R)$ be any ball, let $\LL$ be any affine hyperplane, and fix $\beta > 0$. Let $n$ be the largest integer such that $K^n < 1/\beta$, and let $r = \beta R$ and $\alpha = -\log(1 - \varepsilon)/\log(K) > 0$. Then
\begin{equation}
\label{Knbeta}
R + \frac{2}{\gamma}\sum_{j = 0}^{n - 1} K^j r \leq \left(1 + \frac{2}{\gamma} \frac{K^n}{K - 1} \beta\right)R \asymp R,
\end{equation}
so
\begin{align*}
\mu(B\cap \thickvar{\LL}{\beta R}) = \mu(S_1) &\leq (1 - \varepsilon)^n \mu(S_2)\\
&\leq (K^n)^{-\alpha} \mu\left(B\left(\pp,C R\right)\right)\\
&\asymp \beta^\alpha \mu(B(\pp,R)),
\end{align*}
where $C$ is the implied constant of \eqref{Knbeta}. Thus $\mu$ is absolutely decaying.
\end{proof}

\begin{proof}[Proof of Theorem \ref{theoremsemihyperbolic}]
By \cite[Theorem 2.1]{CJY}, a semi-hyperbolic function has no recurrent critical points or parabolic points in its Julia set. So by \cite[Theorem 1.1]{Urbanski6}, the Hausdorff measure of $J_f$ in dimension $\delta_f$ is positive and finite. Actually, a closer inspection of the proof of \cite[Theorem 1.1]{Urbanski6} shows that the Hausdorff $\delta$-dimensional measure on $J_f$, which we will denote by $\mu$, is Ahlfors $\delta$-regular (simply skip the steps in the proofs of \cite[Lemma 7.12 and Lemma 7.13]{Urbanski6} where \cite[Theorems 2.1(2) and 2.2(2)]{Urbanski6} are used, yielding that the conformal measure is Ahlfors $\delta$-regular). Its support, $J_f$, is hyperplane diffuse by Lemma \ref{lemmahyperplanediffuse}, so by Lemma \ref{lemmaHDFimpliesAD}, $\mu$ is absolutely decaying.
\end{proof}

\bibliographystyle{amsplain}

\bibliography{bibliography}

\end{document}